\documentclass[11pt]{amsart}
\usepackage{amscd}
\usepackage{amsmath,amssymb,latexsym}
\usepackage{pb-diagram}
\usepackage{enumerate}
\usepackage{amsfonts}

\newtheorem{Thm}[equation]{Theorem}
\newtheorem{Cor}[equation]{Corrolary}

\newtheorem{Prop}[equation]{Proposition}
\newtheorem{Lem}[equation]{Lemma}
\newtheorem{Remark}[equation]{Remark}
\newtheorem{Def}[equation]{Definition}
\newcommand{\C}{\mathbb C}
\newcommand{\A}{\mathbb A}

\newcommand\Ker{\operatorname{Ker}}
\newcommand\Span{\operatorname{Span}}
\newcommand\Rep{\operatorname{Rep}}
\newcommand\ind{\operatorname{ind}}
\newcommand\Ind{\operatorname{Ind}}
\newcommand\tr{\operatorname{tr}}
\newcommand\Hom{\operatorname{Hom}}

\newcommand\PGL{\operatorname{PGL}}

\newcommand{\diag}{\operatorname{diag}}

\theoremstyle{plain}
\title{The twisted Satake isomorphism and Casselman-Shalika Formula}
\author{Nadya Gurevich}
\address{School of Mathematics, Ben Gurion University of the Negev, POB 653, Be'
er Sheva 84105, Israel}
\email{ngur@math.bgu.ac.il}
\numberwithin{equation}{section}
\begin{document}
\maketitle

\begin{abstract}
For an arbitrary split adjoint group
we identify the unramified Whittaker space with the space
of skew-invariant functions on the lattice of coweights and 
deduce from it the Casselman-Shalika formula.
\end{abstract} 
\section{Introduction}

The Casselman-Shalika formula is a beautiful formula relating values of special functions
 on a $p$-adic group to the values of finite dimentional complex representations of its
dual group. Further, the formula is particularly useful in the theory of automorphic forms 
for studying $L$-functions.

In this note we provide a new approach to (and a new proof of) Casselman-Shalika formula
for the value of spherical Whittaker functions.

To state our results we fix some notations. Let $G$ be a split adjoint group
over a local field $F$. 
We fix a Borel subgroup $B$ with unipotent radical  $N$, 
and consider its Levi decomposition $B=NT$,
 where $T$ is the maximal split torus.
Furthermore, we fix a maximal compact subgroup $K.$

Let $\Psi$ be a non-degenerate complex character of $N$.
For an irreducible  representation $\pi$ of  $G$ it is well known that 
$\dim \Hom_{G}(\pi,\ind^G_N \Psi) \leq 1$ and in case it is non-zero we say that 
the representation $\pi$ is {\sl generic}. 
%s at most one dimen
The Whittaker model of such a generic irreducible representation $\pi$ of
 $G$ is the image of an embedding  $$W:\pi\hookrightarrow \Ind^G_N \Psi.$$

Let now $\pi$ be a generic irreducible representation. 
Recall that $\pi$ is called {\sl  unramified} if $\pi^{K} \ne \{0\}$ and that in this case $\pi^{K}=\C \cdot v_{0}.$ 
Here $v_0$ is a spherical vector. The explicit formula for the function
$W(v_0)$ was given in [CS] and is commonly called {\sl the Casselman-Shalika formula.} 

Recall that there is a bijection between irreducible unramified
 representations of $G$ and the spectrum of the spherical Hecke algebra $H_K=C_c(K\backslash G/K).$
This commutative algebra admits the following description. 
Let $\Lambda$ be the coweight lattice of $G.$  Recall that $\Lambda$  is canonically identified with $T/(T\cap K)$. 
The Weyl group $W$ acts naturally on the lattice $\Lambda$. 
Denote by $\C[\Lambda]^W$ the algebra of $W$-invariant
elements in $\C[\Lambda]$. 

The Satake isomorphism 
$$S:H_K \simeq \left(\ind^T_{T\cap K} 1\right)^W=\C[\Lambda]^W$$
is defined by
$$S(f)(t)=\delta_B^{-1/2}(t)\int\limits_{N} f(nt)\,dn$$

The main result of the paper is a description 
of  the  Whittaker spherical space $\left(\Ind^G_N \Psi \right)^K$
as a concrete $H_{K} \simeq\C[\Lambda]^{W}$-module. Namely, it is identified with the space 
$\C[\Lambda]^{W,-}$ of functions
on the lattice of coweights, that are skew-invariant under the action of the Weyl group $W$.
More formally, 
\begin{Thm}
 There is a canonical  isomorphism  
$$j:\left(\ind^G_N \Psi\right)^K\rightarrow  \C[\Lambda]^{W,-},$$ 
%\quad j(\phi_\mu)=r_\mu,$$
compatible with Satake isomorphism  $S:H_K\simeq \C[\Lambda]^W.$
\end{Thm}

From this result it easily follows that the twisted Satake map
$$S_\Psi: C_c(G/K) \rightarrow \left(\ind^G_N \Psi\right)^K, \quad
S_\Psi(f)(t)=\int\limits_{N} f(nt)\overline{\Psi(n)} \,dn$$
sends  the spectral basis of the spherical Hecke algebra
$H_K=C_c(K\backslash G/K)$
to the basis of characteristic functions of $\left(\ind^G_N \Psi\right)^K$.
In [FGKV], it is explained that this latter result is equivalent 
to the Casselman-Shalika formula in [CS] (see section \ref{CS:section} for a full account). Thus, we obtain 
a proof of the Casselman-Shalika formula that does not 
use the uniqueness of the Whittaker model. 
\vskip 10pt

Let us now quickly describe our proof of the main result. 
It was inspired by a new simple proof \cite{S} by Savin of the Satake
isomorphism of algebras
$$S:H_K \simeq \left(\ind^T_{T\cap K} 1\right)^W=\C[\Lambda]^W.$$
Savin has observed that the Satake map 
$S: C_c(G/K) \rightarrow \left(\ind^G_N 1\right)^K$ 
restricted to $$C_c(I\backslash G/K)=\left(\ind^G_I 1\right)^K,$$ 
where $I$ is an Iwahori
subgroup, defines an explicit isomorphism
$$\left(\ind^G_I 1\right)^K \simeq \ind^T_{T\cap K} 1=\C[\Lambda].$$

Restricting $S_\Psi$ to $\left(\ind^G_I 1 \right)^K$, we prove 
there there exists an isomorphism $$j:\left(\ind^G_N \Psi\right)^K\rightarrow  \C[\Lambda]^{W,-}$$
of $H_K\simeq \C[\Lambda]^W$-modules,
making the following diagram

$$
\begin{diagram}
\node{\left(\ind^G_I 1 \right)^K}
\arrow{e,t}{S_\Psi}
\arrow{s,l}{S}
\node {\left(\ind^G_N \Psi\right)^K}
\arrow{s,r}{j} \\
%\node{\left(\ind^T_{T\cap K} 1\right)}
\node{\C[\Lambda]}
\arrow{e,t}{alt} 
%\node{\left(\ind^T_{T\cap K} 1\right)^{W,-}}
\node{\C[\Lambda]^{W,-}}
\end{diagram}
$$
commutative. Here the space $\C[\Lambda]^{W,-}$ is a space of 
$W$ skew-invariant elements of $\C[\Lambda]$ 
and the alternating  map $alt$ are defined in the section \ref{func:lattices}.

\vskip 5pt
{\bf Acknowlegdements.}
I wish to thank Gordan Savin for explaining his proof of the Satake isomorphism
and encouraging me to prove the Casselman-Shalika formula. I am most grateful
to Joseph Bernstein for his attention and his help in formulating
the main result. I thank Roma Bezrukavnikov for answering my questions
and Eitan Sayag for his suggestions for improving the presentation.
%The main part of this work was done during the workshop
%in RIMS. I wish to thank the organizers for their hospitality. 
My research is partly supported by ISF grant $1691/10$.

\section{Notations}
Let $F$ be a local non-archimedean field and let $q$ be a characteristic of 
its residue field.
Let $G$ be a split adjoint group defined over $F$. 
Denote by $B$ a Borel subgroup of $G$, by $N$ its unipotent radical,
by $\bar N$ the opposite unipotent radical,
 by $T$ the maximal split torus and  by $W$  the Weyl group.

Denote by $R$ the set of positive roots of $G$
and by $\Delta$ the set of simple roots.
For each $\alpha\in R$ let 
$x_\alpha:F\rightarrow N$ denote the one parametric subgroup 
corresponding to the root $\alpha$ and 
$N^k_\alpha=\{x_\alpha(r):|r|\le q^{-k}\}$.

Let $\Psi$ be a non-degenerate complex character of $N$ of conductor $1$,
i.e. for any $\alpha\in \Delta$ 
$$\Psi|_{N^{0}_\alpha}\neq 1, \Psi|_{N^1_\alpha}=1.$$

Let $K$ be a maximal compact subgroup of $G$.
% and $I\subset K$ be an
%Iwahori subgroup such that $I\cap B=K\cap B$. 
Then $T_K=T\cap K$ is a maximal compact subgroup of $T$.
Choose an Iwahori subgroup $I\subset K$ such that
$I\cap N=N_\alpha^1$ for all $\alpha \in R$.
In particular $\Psi_{N\cap I}=1$, but 
$\Psi|_{N^0_\alpha}\neq 1$ for any $\alpha \in \Delta$.

We fix a Haar measure on $G$ normalized such that 
the measure of $I$ is one.

The coweight lattice $\Lambda$
of $G$ is canonically identified with $T/T_K.$
For any $\lambda\in \Lambda$ denote by  $t_\lambda\in T$ its representative.
The coweight $\rho$ denotes the half of all the positive coroots.
Since $G$ is adjoint one has $\rho\in \Lambda$.
We denote by $\Lambda^+$ the set of dominant coweights.
\iffalse
 and for any 
$k\ge 0 $ define the set of coweights $\Lambda^{+,k}\subset \Lambda^+$ by
$$\Lambda^{+,1}=\{\mu\in \Lambda: 
\forall \alpha \in \Delta, (\alpha,\mu)\ge 1\}
$$
\fi

Let $^LG$ be the complex dual group of $G$.
Then $\Lambda$ is also identified with the lattice of weights of $^LG$.
For a dominant weight $\lambda$ we denote
by  $V_\lambda$  the highest weight module
of $^LG$ and by $wt(V_\lambda)$ the multiset of all the weights of this module.

\section{Functions on lattices}\label{func:lattices}

Consider the algebra $\C[\Lambda]=Span\{ e^\nu: \nu \in \Lambda\}.$
The Weyl group $W$ acts naturally on the lattice $\Lambda$. 
We denote by $\C[\Lambda]^W$ the algebra of $W$-invariant
elements in $\C[\Lambda]$.
The character map defines an isomorphism of algebras
$$Rep(^LG)\simeq \C[\Lambda]^W.$$
For an irreducible module $V_\lambda$ denote
$$a_\lambda=char(V_\lambda)=\sum_{\nu\in wt(V_\lambda)} e^\nu.$$
The elements $\{a_\lambda|\lambda\in \Lambda^+\}$ 
form a basis of $\C[\Lambda]^W$.
The algebra $\C[\Lambda]^W$ acts on the space $\C[\Lambda]$ by multiplication. 
\vskip 10pt
%Now consider the twisted action of $W$ on $\C[\Lambda]$ by
%$$w(e^\lambda)=e^{w(\lambda+\rho)-\rho}.$$
The element $f\in \C[\Lambda]$ is called {\sl skew-invariant}
if $w(f)=(-1)^{l(w)}f$, where $l(w)$ is the length of the element $w$.
Denote by  $\C[\Lambda]^{W,-}$ the space of $W$ skew-invariant elements.
 The algebra $\C[\Lambda]^W$
acts on $\C[\Lambda]^{W,-}$ by multiplication. Note that the action
is torsion free. 

Define the alternating map 
$$alt:\C[\Lambda] \rightarrow \C[\Lambda]^{W,-},\quad 
alt(e^\mu)=\sum_{w\in W} (-1)^{l(w)}e^{w\mu}: \quad \mu\in \Lambda$$
It is a map of $C[\Lambda]^W$ modules.
The elements  
 $$\{r_{\mu+\rho}=
\sum_{w\in W} (-1)^{l(w)}e^{w(\mu+\rho)}: \quad \mu\in \Lambda^+\}$$
form a basis of $\C[\Lambda]^{W,-}.$
Note that for any $\lambda\in \Lambda^+$
 $$alt(e^{\lambda+\rho})=r_{\lambda+\rho}=r_\rho\cdot a_\lambda=alt(a_\lambda),$$
where the second equality is the Weyl character formula.

\section{Hecke algebras}

\subsection{The spherical Hecke algebra}
The  spherical Hecke algebra $H_K=C_c(K\backslash G/K)$
is the algebra of locally constant compactly supported 
bi-$K$ invariant functions with the multiplication given by
convolution $\ast$. It has identity element $1_K$ - the characteristic
function of $K$ divided by $[K:I]$. 

Consider the Satake map  
$$S: C_c(G/K)\rightarrow C(N\backslash G/K)=C(T/T_K)$$ defined by
$$S(f)(t)=\delta_B^{-1/2}(t)\int\limits_{N} f(nt)dn.$$
The famous Satake theorem claims that the restriction
of $S$ to $H_K$ defines an isomorphism 
of algebras $S:H_K\simeq \C[\Lambda]^W$.
Denote by $A_\lambda$ the element of  $H_K$ corresponding to $a_\lambda$
under this map. Thus $H_K=\Span\{ A_\lambda: \lambda\in \Lambda^+\}$.

%For any unramified representation $\pi$ with the Satake parameter
%$t_\pi\in ^LT/W$ and spherical vector $v^0$ one has 
%$$\int\limits_{G} A_\lambda(g)\pi(g)v^0 dg=tr_{V_\lambda}(t_\pi)$$

\subsection{The Iwahori-Hecke algebra}\label{IH:algebra}

The Iwahori-Hecke algebra $H_I=C_c(I\backslash G/I)$
is the algebra of locally constant compactly supported 
bi-$I$ invariant functions with the multiplication given by
convolution. Below we remind the list of properties of $H_I$,
all can be found in [HKP].
\begin{enumerate}
\item
The algebra $H_I$ 
contains a commutative algebra 
$A\simeq \C[\Lambda]$. 
$$A=\Span\{\theta_\mu|\,\, \mu \in \Lambda\},$$ where
$$\theta_\mu=\left\{
\begin{array}{ll} \delta^{1/2}_B 1_{I t_\mu I} & \mu\in \Lambda^+;\\
\theta_{\mu_1}\ast\theta_{\mu_2}^{-1}& \mu =\mu_1-\mu_2,
\quad  \mu_1,\mu_2\in \Lambda^+
\end{array}
\right..
$$
The center $Z_I$ of the algebra $H_I$ is $A^W\simeq \C[\Lambda]^W$.
\item
The finite dimensional Hecke algebra $H_f=C(I\backslash K/I)$ is 
a subalgebra of $H_I$. The elements $t_w=1_{IwI}$ , where $w\in W$
form a basis of $H_f$. Multiplication in $H$ induces a vector space
isomorphism 
$$H_f\otimes_\C A\rightarrow H_I$$
In particular the elements $t_w\theta_\mu$ where $w\in W, \mu \in \Lambda$
form a basis of $H_I$.
\item
The algebra $H_K$ is embedded naturally in $H_I$. 
One has $H_K=Z_I\ast 1_K$  and 
$$A_\lambda=\left(\sum_{\nu \in wt(V_\lambda)}\theta_\nu\right)\ast 1_K$$ 
\item
For the simple reflection $s\in W$ corresponding to 
a simple root $\alpha$ and a coweight $\mu$ one has 
$$t_s\theta_\mu=\theta_{s\mu}t_s+ 
(1-q) \frac{\theta_{s\mu}-\theta_{\mu}}{1-\theta_{-\alpha}}.$$

In particular $t_s$ commutes with $\theta_{k\alpha}+\theta_{-k\alpha}$
for any $k\ge 0$.
In addition $t_s$ commutes with $\theta_\mu$ whenever $s\mu=\mu$.
\end{enumerate}

\subsection{The intermediate algebra}
Finally consider the space $H_{I,K}$ defined by
$$H_K\subset H_{I,K}=H_I\ast 1_K=C_c(I\backslash G/K)\subset H_I.$$ It has a structure of 
right $H_K$ module. The  space $H_{I,K}$ plays a crucial role in Savin's paper [S].
The Satake map restricted to it is the isomorphism of 
$H_K\simeq \C[\Lambda]^W$ modules:
$$S:H_{I,K}\simeq \C[\Lambda],\quad S(\theta_\mu\ast 1_K)=e^\mu.$$
In particular, it is shown that the elements 
$\{\theta_\mu^K=\theta_\mu\ast 1_K,\,  \mu\in\Lambda\}$
form a basis of $H_{I,K}.$

\section {The Whittaker space $\left(\ind^G_N \Psi\right)^K$} 
Let $\Psi$ be a non-degenerate character of conductor $1$.
Consider the space $\left(\ind^G_N \Psi\right)^K$ 
of complex valued functions on $G$ that
are  $(N,\Psi)$-equivariant on the left,
right $K$-invariant functions and are compactly supported modulo 
$N$.

The space $\left(\ind^G_N \Psi\right)^K$ 
has a structure of right $H_K$ module by 
$$(\phi\ast f)(x)=\int\limits_{G} \phi(xy^{-1})f(y) dy, 
\quad \phi \in \left(\ind^G_N \Psi\right)^K, f\in H_K.$$

Any function $\phi$ on $\left(\ind^G_N \Psi\right)^K$ 
is determined by its values on 
$t_\lambda :\lambda\in \Lambda$
and $\phi(t_\lambda)=0$ unless $\lambda\in \Lambda^+\rho$.

The space $\left(\ind^G_N \Psi\right)^K$ 
has a basis of characteristic functions 
$\{\phi_\lambda: \lambda \in \Lambda^++\rho\}$ where
 $$\phi_\lambda(ntk)=\left\{
\begin{array}{ll}\delta^{1/2}_B(t)\Psi(n) & t\in Nt_\lambda K
\quad \lambda\in \Lambda^++\rho ;\\
0 & {\rm otherwise}
\end{array}
\right..$$

The main theorem of this paper is the description of
$\left(\ind^G_N \Psi\right)^K$ as $H_K$ module.

\begin{Thm}\label{main}
Let $\Psi$ be a character of conductor $1$.
Then there is an isomorphism 
$$j:\left(\ind^G_N\Psi\right)^K \simeq \C[\Lambda]^{W,-}$$ 
compatible with $H_K\simeq \C[\Lambda]^W$.
\end{Thm}

\iffalse
\begin{Thm} \label{main} Let $\Psi$ 
 be a non-degenerate character of conductor $0$. The map 
$$j:\left(\ind^G_N \Psi\right)^K\rightarrow  \C[\Lambda]^{W,-}, 
\quad j(\phi_\mu)=r_\mu,$$
is an isomorphism of   $H_K\simeq \C[\Lambda]^W$-modules.
\end{Thm}

%The theorem will be proved in the section \ref{main:proof}.
\fi
\subsection{ The twisted Satake isomorphism}

For a fixed character $\Psi$ of $N$, consider a twisted Satake map
$$S_\Psi: C_c(G/I)\rightarrow \left(\ind^G_N \Psi\right)^I$$ defined by
$$S_\Psi(f)(t)=\int\limits_{N} f(nt)\overline{\Psi(n)} \,dn.$$

\begin{Cor}\label{CS2}
The restriction of $S_\Psi$ to the  right $H_K$ submodule  
$\theta^K_\rho\ast H_K$ defines an  isomorphism 
 $$S_\Psi:\theta^K_\rho\ast H_K\simeq \left(\ind^G_N \Psi\right)^K$$ such that 
 $S_\Psi(\theta^K_\rho\ast  A_\lambda)=\phi_{\lambda+\rho}$.
\end{Cor}

\begin{proof} 
By Weyl character formula $r_{\lambda+\rho}=r_\rho\cdot a_\lambda$.
Hence 
$$j(\phi_{\lambda+\rho})=r_{\lambda+\rho}=r_\rho\cdot a_\lambda =
j(\phi_\rho\ast A_\lambda),$$
and thus $\phi_{\rho}\ast A_\lambda=\phi_{\lambda+\rho}.$

Restricting  $S_\Psi$ to $H_{I,K}$, we obtain
$$S_\Psi(\theta^K_\rho\ast A_\lambda)=S_\Psi(\theta^K_\rho)\ast A_\lambda=
\phi_\rho\ast A_\lambda=\phi_{\lambda+\rho}.$$ 
 Since $A_\lambda$ and $\phi_{\lambda+\rho}$ are bases
of $H_K$ and $\left(\ind^G_N \Psi\right)^K$ respectively, 
the map $S_\Psi$ is an isomorphism.
\end{proof}

To prove the theorem we shall need two lemmas. The first one 
ensures surjectivity of the map $S_\Psi$ and the second one 
describes its kernel.

\begin{Lem}
$S_\Psi(\theta_\mu^K)=\phi_\mu$ for all $\mu\in \Lambda^++\rho$. 
In particular the map  
$$S_\Psi:H_{I,K}\rightarrow  \left(\ind^G_N\Psi\right)^K$$
is surjective.
\end{Lem}

\begin{proof}
It is enough to compute $S_\Psi(\theta_{\mu}\ast 1_K)(t_\gamma)$ for 
$\gamma\in \Lambda^+.$

Since $\mu$ is dominant one has 
$$\theta^K_\mu= \delta_B^{1/2}(t_\mu)1_{It_\mu K}$$ and hence
 $$S_\Psi(\theta^K_\mu)(t_\gamma)=
\delta_B^{1/2}(t_\gamma)\int\limits_{N_{\gamma,\mu}}\overline{\Psi(n)}\,dn,$$
where 
$$N_{\gamma,\mu}=\{n\in N: n t_\gamma\in I t_\mu K\}.$$

The set 
$$N_{\gamma,\mu}=
\left\{
\begin{array}{ll}\emptyset& \gamma\neq \mu\\
N\cap K & \gamma=\mu
\end{array}
\right.
$$

Indeed, since $\mu\in \Lambda^+$ one has 
 $I t_\mu K = (N\cap I) t_\mu K$. 
One inclusion is obvious. For another inclusion use the Iwahori factorization
$$I=(I\cap N)T_K(I\cap \bar N)$$ to represent any $g\in I t_\mu K$ as
$$ g=n a_0 \bar n t_\mu k = n t_\mu a_0(t_\mu^{-1} \bar n t_\mu) k,$$
where $n\in N\cap I, a_0\in T_K, \bar n\in \bar N, k\in K$.
Since $\mu$ is dominant one has
$(t_\mu^{-1} \bar n t_\mu)\in K$. So $g\in (N\cap I) t_\mu K$.
Hence $N_{\gamma,\mu}=\emptyset$ unless $\gamma=\mu$
and $N_{\mu,\mu}=(N\cap I)t_\mu (N\cap K) t^{-1}_\mu = N\cap I$
since $\mu \in \Lambda^++\rho.$
In particular $\Psi|_{N_{\mu,\mu}}=1$. Hence 
$$S_\Psi(\theta_\mu\ast 1_K)=\phi_\mu.$$
\end{proof}

\begin{Lem} Let 
$\alpha\in \Delta$, $s$ be a simple reflection corresponding
to $\alpha$ and $\iota_\alpha=1_I+t_s$ be the characteristic function of a parahoric 
subgroup $I_\alpha$ corresponding to  $\alpha$.

\begin{enumerate} 
\item   $S_\Psi(\iota_\alpha)=0$.
\item 
$S_\Psi(\theta^K_\mu+\theta^K_{s.\mu})=0$ for all $\mu\in \Lambda$. 
\end{enumerate}
\end{Lem}

\begin{proof}
1) $$S_\Psi(\iota_\alpha)(tw)=
\int\limits_N \iota_\alpha(ntw) \overline{\Psi(n)}\, dn= 
\int\limits_{N\cap I_\alpha(tw)^{-1}} \overline{\Psi(n)}\,dn$$
The set $N\cap I_\alpha(tw)^{-1}$  is empty unless $w\in \{e, s\}$ and $t\in T_K$, in which case
$$S_\Psi(\iota_\alpha)(tw)=\int\limits_{N\cap I_\alpha} \overline{\Psi(n)}\,dn=0$$
since the integral contains an inner integral over $N^0_\alpha$ on which 
$\Psi$ is not trivial.

2) Let us represent any $\mu=\mu'+k\alpha$ where $(\mu',\alpha)=0$.
Then $s\mu=\mu'-k\alpha$.
In particular 
$$\theta^K_\mu+\theta^K_{s\mu}=\theta_{\mu'}
(\theta_{k\alpha}+\theta_{-k\alpha}) \iota_\alpha 1_K$$
By the results in \ref{IH:algebra} the element 
$i_\alpha$ commutes with  $\theta_{\mu'}
(\theta_{k\alpha}+\theta_{-k\alpha})$ and hence the above equals 
$$\iota_\alpha\theta_{\mu'}
(\theta_{k\alpha}+\theta_{-k\alpha}) 1_K=\iota_\alpha (\theta^K_\mu+\theta^K_{s\mu}) $$
 
By part $(1)$ it follows that $S_\Psi(\theta^K_\mu+\theta^K_{s\mu})=0$.
\end{proof}

\begin{proof}  of \ref{main}.
We have shown that the map $S_\Psi$ is surjective onto $(\ind^G_N \Psi)^K$
and 
$$\Ker S_\Psi=\Span\{ \theta_\mu -(-1)^{l(w)}\theta_{w\mu}| \mu \in \Lambda, w\in W\}.$$
Another words 
$$(\ind^G_N \Psi)^K\simeq H_{I,K}/\Ker S_\Psi=\C[\Lambda]^{W,-}$$ as
$H_K\simeq \C[\Lambda]^W$-modules.
\end{proof}

\section{Casselman-Shalika Formula}\label{CS:section}

Let $(\pi,G,V)$ be an irreducible smooth generic unramified representation and 
denote by $\gamma\in {^LT}/W$ its Satake conjugacy class. 
Choose a spherical vector $v_0$ and normalize the Whittaker functional
$W_\gamma\in \Hom_{G}(\pi, \Ind^G_N \bar \Psi)$ such that $W_\gamma(t_\rho v_0)=1$.

The Casselman-Shalika formula reads as follows:

\begin{Thm}\label{CS1}
$$W_\gamma(v_0)(t_{\lambda+\rho})=
 \left\{
\begin{array}{ll}\delta_B^{1/2}(t_{\lambda+\rho}) 
\tr V_\lambda(t_\gamma)& \lambda \in \Lambda^+\\
0 & {\rm otherwise}
\end{array}
\right.$$  
\end{Thm}

It is shown in [FGKV], that Theorem  $5.2$ that the formula 
(\ref{CS1}) implies the  Corollary \ref{CS2} and it is 
mentioned that the two statements are equivalent.
Let us now prove the other direction.

\begin{proof}
We deduce the formula $\ref{CS1}$ from \ref{CS2}.
Let $\pi$ be a generic unramified representation with the Satake parameter 
$\gamma\in {^LT}$ and a spherical vector $v_0$ and the Whittaker model
\hbox{$W_\gamma:\pi_\gamma\rightarrow \Ind^G_N\bar \Psi$} such that 
$W_\gamma(v_0)(t_\rho)=1$.  Define the map
$\chi_\gamma: H_K\rightarrow \C$ by 
$$\pi(f)v_0=\int\limits_{G} f(g)\pi(g)v_0\, dg =\chi_\gamma(f) v_0$$
and the map $r_\gamma: \ind^G_N\Psi\rightarrow \C$ by 
$$r_\gamma(\phi)=\int\limits_{N\backslash G} W_\gamma(v_0)(g)\phi(g)\, dg.$$

Then 
$$r_\gamma(S_\Psi(\theta^K_\rho\ast A_\lambda))=
\int\limits_{N\backslash G}\int\limits_{N} (\theta^K_\rho\ast A_\lambda)(ng)\bar \Psi(n) W_\gamma(v_0)(g) \, dn\, dg=$$
$$\int\limits_{G} \int\limits_{G}(\theta^K_\rho\ast A_\lambda)(g) W_\gamma(v_0)(g) \, dg=$$
$$\int\limits_{G}\int\limits_{G} \theta^K_\rho(gx^{-1}) A_\lambda(x) W_\gamma(gx^{-1}\cdot x \cdot v_0)(1)\, dx \, dg=
\chi_\gamma(A_\lambda) W_\gamma(v_0)(t_\rho).$$

%It is easy to see that for any $f\in H_K$,  $r_\gamma(S_\Psi(f))=\chi_\gamma(f).$

Under the identification $H_K\simeq \Rep(^LG)$ the homomorphism
$\chi_\gamma$ sends an irreducible representation $V$ to $\tr V(\gamma)$.
In particular $\chi_\gamma(A_\lambda)=\tr V_\lambda(\gamma)$.
$$\tr V_\lambda(\gamma)=\chi_\gamma(A_\lambda)W(v_0)(t_\rho)=$$
$$r_\gamma(S_\Psi(\theta^K_\rho \ast A_\lambda))=r_\gamma(\phi_{\lambda+\rho})=\delta_B^{-1/2}(t_{\lambda+\rho}) 
W_\gamma(v_0)(t_{\lambda+\rho})$$
Hence 
$$W_\gamma(v_0)(t_{\lambda+\rho})=\delta_B^{1/2}(t_{\lambda+\rho})\tr V_\lambda(\gamma).$$
\end{proof}


\begin{thebibliography}{99}

\iffalse
% \bibitem[AJ]{AJ} J. Adams and J. F. Johnson, 
%{\em Endoscopic groups and packets of nontempered representations}, 
%Compositio Math. 64 (1987), no. 3, 271--309. 
%\bibitem[AB]{AB} S. Arkhipov, R. Bezrukavnikov 
%{\em Perverse sheaves on affine flags and Langlands dual group. 
%With an appendix by Bezrukavrikov and Ivan Mirković.} Israel J. Math. 170 (2009), 135-183.

\bibitem[BK]{BK} R. Bezrukavnikov, D. Kazhdan {\em Geometry of second adjointness for p-adic groups } arXiv:1209.0403. 
\fi
%\bibitem[Bu]{Bu} Bump

\bibitem[CS]{Ca} W. Casselman, J. Shalika
 {\em The unramified principal series of p-adic groups. II. The Whittaker function.} Compositio Math. 41 (1980), no. 2, 207–231.

\bibitem [FGKV]{FGKV} E. Frenkel, D. Gaitsgory, D. Kazhdan, K. Vilonen
{\em  Geometric realization of Whittaker functions and the Langlands conjecture.}
 J. Amer. Math. Soc. 11 (1998), no. 2, 451-484. 


\bibitem [HKP] {HKP} T. Haines, R. Kottwiz, A. Prasad 
{\em Iwahori-Hecke Algebras},
 J. Ramanujan Math. Soc. 25 (2010), no. 2, 113-145. 

%\bibitem[PS]{PS} Piatetskii-Shapiro, Rallis

\bibitem [S]{S} G. Savin {\em The tale of two Hecke algebras},  arXiv:1202.1486 








\end{thebibliography}
\end{document}